\documentclass[11pt]{article}

\usepackage{amsfonts,amsmath,amsthm,latexsym,color,epsfig,mathrsfs,enumerate}

\newtheorem{thm}{Theorem}
\newtheorem{dfn}[thm]{Definition}
\newtheorem{lem}[thm]{Lemma}
\newtheorem{cor}[thm]{Corollary}
\newtheorem{prop}[thm]{Proposition}

\newcommand{\A}{\mathcal{A}}
\newcommand{\B}{\mathcal{B}}
\newcommand{\C}{\mathcal{C}}
\newcommand{\D}{\mathcal{D}}

\newcommand{\Z}{\mathbb{Z}}
\newcommand{\N}{\mathbb{N}}

\DeclareMathOperator{\diam}{diam}

\title{Frankl-R\"odl type theorems for codes and permutations}
\author{Peter Keevash \thanks{\mbox{Mathematical Institute, Oxford OX2 6GG, UK. Email: Peter.Keevash@maths.ox.ac.uk.} Research supported in part by ERC grant 239696 and EPSRC grant EP/G056730/1.} \and Eoin Long \thanks{Mathematical Institute, Oxford OX2 6GG, UK. Email: Eoin.Long@maths.ox.ac.uk.}}

\begin{document}

\maketitle

%
%
%
%
%
%
%
%

\begin{abstract}
We give a new proof of the Frankl-R\"odl theorem on forbidden intersections, via the probabilistic method of dependent random choice. Our method extends to codes with forbidden distances, where over large alphabets our bound is significantly better than that obtained by Frankl and R\"odl. We also apply our bound to a question of Ellis on sets of permutations with forbidden distances, and to establish a weak form of a conjecture of Alon, Shpilka and Umans on sunflowers.
\end{abstract}

%
%
%
%
%
%
%
%

\section{Introduction}

\let\thefootnote\relax\footnote{\emph{2010 Mathematics Subject Classification.} Primary 05D05. Secondary 05D40, 94B65.}

A family $\A $ of sets is said to be \emph{$l$-avoiding} if $|A \cap B| \neq l$ for all $A, B \in \A $. Erd\H{o}s conjectured (\cite{Erd}) that for any $\epsilon \in (0,1)$ there is $\delta = \delta (\epsilon ) > 0$ such that given $l$ with $\epsilon n \leq l \leq (1/2 - \epsilon )n$, any $l$-avoiding family ${\mathcal A} \subset {\mathcal P}[n]$ satisfies $|\A | \leq (2-\delta )^n$ and offered $\$ 250$ for a solution. In \cite{FrRo}, Frankl and R\"odl gave a positive answer to Erd\H{o}s' conjecture, proving the following stronger result:

\begin{thm}[Frankl-R\"odl]
\label{FR}
Let $\alpha, \epsilon \in (0,1)$ with $\epsilon \leq \alpha /2$. Let $k = \lfloor \alpha n\rfloor $ and $l\in [\max (0, 2k-n) +\epsilon n, k - \epsilon n]$. Then any $l$-avoiding family ${\mathcal A} \subset \binom {[n]}{k}$ satisfies $|\A | \leq (1 - \delta )^n\binom {n}{k}$ where $\delta = \delta (\alpha ,\epsilon )>0$.
\end{thm}
\noindent Theorem \ref{FR} along with several extensions of the theorem proved in \cite{FrRo} have had a huge impact in a number of different areas including discrete geometry \cite{FRGeomRams}, communication complexity \cite{Sgall} and quantum computing \cite{BCW}.

In Section 2 of this paper we give a new proof of Theorem \ref{FR}. We show that the theorem can in fact be deduced from an earlier theorem due to Frankl and Wilson (see Theorem \ref{FW} below). While our new proof of Theorem \ref{FR} does not seem to improve on the bounds given in \cite{FrRo}, the same proof method does significantly improve bounds when we forbid distances over a larger underlying alphabet. Given $q \in {\mathbb N}$, $q\geq 2$, we will say that a set $\mathcal C $ is a $q$-ary code if $\mathcal C \subset [q]^n$. The Hamming distance between two words $x,y \in [q]^n$ is written as $d_H(x,y) = |\{i\in [n]: x_i \neq y_i\}|$. For a code $\mathcal C$ we write $d({\mathcal C} ) = \{d_H(x,y): \mbox{ distinct }x,y \in {\mathcal C } \} \subset [n]$. Frankl and R\"odl used Theorem \ref{FR} to prove the following result:

\begin{thm}[Frankl-R\"odl]
	\label{FRcodetheorem}
 Let ${\mathcal C} \subset [q]^n$, and let $\epsilon $ satisfy $0 < \epsilon < 1/2$. Suppose that $\epsilon n < d < (1 - \epsilon )n$, and $d$ is even if $q = 2$. If $d \notin d({\mathcal C})$, then $|{\mathcal C} | \leq (q - \delta )^n$ with some positive constant $\delta = \delta (\epsilon , q)$.
\end{thm}

\noindent (Note that, in order for Theorem \ref{FRcodetheorem} to hold for $q=2$, we must have that $d$ is even since the set ${\C }_0 = \{x\in \{0,1\}^n: \sum _{i} x_i \equiv 0 \pmod 2\}$ satisfies $|{\mathcal C_0}| = 2^{n-1}$ but contains no odd distances.)

In Section 3, we improve this to the following:

\begin{thm}
 \label{newcodebound}
 Let ${\mathcal C} \subset [q]^n$, and let $\epsilon $ satisfy $0 < \epsilon < 1/2$. Suppose that $\epsilon n < d < (1 - \epsilon )n$, and $d$ is even if $q = 2$. If $d \notin d({\mathcal C})$, then
$|{\mathcal C } | \leq q^{(1-\delta )n}$ with some positive constant $\delta = \delta (\epsilon )$.
\end{thm}

As a consequence of Theorem \ref{newcodebound} we obtain a Frankl-R\"odl type theorem for permutations. Given two permutations $\pi ,\rho \in S_n$ we write 
\begin{equation*}
d_{S_n}(\pi ,\rho ) = |\{i \in [n]: \pi (i) \neq \rho (i)\}|.
\end{equation*} 
For a set ${\mathcal S} \subset S_n$ we write 
\begin{equation*}
d_{S_n}({\mathcal S}) = \{d\in [n]: d(\pi ,\rho ) = d \mbox{ for distinct } \pi , \rho \in {\mathcal S}\}.
\end{equation*} 
Recently Ellis \cite{Ell} asked how large a family ${\mathcal S} \subset S_n$ can be if $d\notin d_{S_n}({\mathcal S})$ for some $d\in [n]$. A result of Deza and Frankl \cite{DF} answers this question for $d = n$, showing that the largest such families have size $(n-1)!$. Ellis \cite{Ell} gave a tight upper bound of $(n-2)!$ when $d = n-1$, provided $n$ is sufficiently large. Here we consider this question when $\epsilon n < d < (1-\epsilon )n$ for $\epsilon >0$. It is easily seen that for such $d$ there exist sets of permutations ${\mathcal S} \subset S_n$ with $d\notin d_{S_n}({\mathcal S} )$ such that $|{\mathcal S} | \geq (n!)^c$ where $c = c(\epsilon ) \in (0,1)$. By taking $q=n$ and viewing permutations $\pi \in S_n$ as vectors in $[q]^n$, with $\pi = (\pi (1),\ldots ,\pi (n))$, since $|S_n| = n! = q^{(1-o(1))n}$, Theorem \ref{newcodebound} has the following consequence:

\begin{thm}
 \label{FRpermut}
 Let ${\mathcal S} \subset S_n$, and let $\epsilon $ satisfy $0 < \epsilon < 1/2$. Suppose that $\epsilon n < d < (1 - \epsilon )n$. If $d \notin d_{S_n}({\mathcal S})$, then
$|{\mathcal S}| < (n!)^{(1-\delta )}$ with some positive constant $\delta = \delta (\epsilon)$.
\end{thm}

Before we discuss another consequence of Theorem \ref{newcodebound}, we need the following definition.

\begin{dfn}
Given $v,w \in [q]^n$ let $\mbox{Agree}(v,w) = \{i \in [n]: (v)_i = (w)_i\}$. A collection of vectors $v_1,\ldots ,v_k \in [q]^n$ is said to form a \textit{strong sunflower} with $k$ petals in $[q]^n$ if there is a fixed set $S \subset [n]$ such that $\mbox{Agree}(v_i,v_j) = S$ for all distinct $i,j \in [k]$. A collection of vectors $v_1,\ldots ,v_k \in [q]^n$ is said to form a \textit{weak sunflower} with $k$ petals in $[q]^n$ if there is $D \in {\mathbb N}$ such that $|\mbox{Agree}(v_i,v_j)| = D$ for all distinct $i,j \in [k]$. 
\end{dfn}

Using Theorem \ref{FR}, Frankl and R\"odl proved that for any $k\in {\mathbb N}$ there exists $\delta = \delta (k) > 0$ such that if ${\mathcal A} \subset \{0,1\}^n$ with $|{\mathcal A}| > (2-\delta )^n$ then ${\mathcal A}$ contains a weak sunflower with $k$ petals. Similarly, using the methods from \cite{FrRo} it can be shown that for any $k\in {\mathbb N}$ there exists $\delta = \delta (q,k) >0$ so that given a code ${\mathcal C} \subset [q]^n$ with $|{\mathcal C}| \geq (q-\delta )^n$, ${\mathcal C}$ contains a weak $k$-petal sunflower in $[q]^n$. In Section 4 we prove the following:

\begin{thm} 
 \label{weaksunflowersfor[q]}
Given $k \in \N$, there exists $\delta = \delta (k) >0$ such that the following holds. For $q\geq 2$, every ${\mathcal C} \subset [q]^n$ which does not contain a weak sunflower with $k$ petals satisfies $|{\mathcal C}| \leq q^{(1-\delta )n}$.
\end{thm}

\noindent This might be seen as giving evidence to a recent conjecture of Alon, Shpilka and Umans \cite{ASU} who asked for a similar bound on families not containing a strong sunflower with $3$ petals in $[q]^n$.

A crucial idea in the original proof of Theorem \ref{FR}, along with an ingenious density increment argument, was to prove a stronger result. In \cite{FrRo} the authors actually proved a cross-intersecting version of Theorem \ref{FR}:
\begin{thm}[Frankl-R\"odl]
\label{crossFR}
Let $\alpha, \epsilon \in (0,1)$ with $\epsilon \leq \alpha /2$. Let $k = \lfloor \alpha n\rfloor $ and $l\in [\max (0, 2k-n) +\epsilon n, k - \epsilon n]$. Then if ${\mathcal A}_1, {\mathcal A}_2 \subset \binom {[n]}{k}$ with $|A _1 \cap A_2| \neq l$ for all $A_i \in {\mathcal A}_i$, they satisfy $|\A _1||\A _2| \leq (1 - \delta )^n\binom {n}{k}^2$, where $\delta = \delta (\alpha ,\epsilon )>0$.
\end{thm}
We draw attention to the fact that the corresponding cross versions of Theorem \ref{newcodebound} and Theorem \ref{FRpermut} with our improved bounds do not hold in general. Indeed, for even $n$, if we take $\A _1$ to be the collection of all permutations in $S_n$ sending $[n/2]$ to $[n/2]$ and $\A _2$ to be the collection of all permutations in $S_n$ sending $[n/2]$ to $[n/2+1 ,n]$ we see that $|\A _i| \geq (n/2)!^2 \geq n!/3^n = (n!)^{1-o(1)}$ but $d_{S_n}(\rho _1,\rho _2) = n$ for all $\rho _i \in \A _i$. 

However, in Section 5 we give a simple condition which guarantees fixed distances between such sets.

\begin{thm}
 \label{crosscodedistance}
Given $\epsilon \in (0,1/2)$ there exists $\delta ', \gamma >0$ such that the following holds. Let $q\geq 3$ and suppose that $\C ,\D \subset [q]^n$ with $|\C |\geq q^{(1-\delta ')n}$ and such that for all $x\in \C$ there exists $y\in \D $ with $d_H(x,y) \leq \gamma n$. Then given any $d\in (\epsilon n, (1-\epsilon )n)$, there exists $x\in \C $ and $y\in \D $ with $d_H(x,y) = d$.
\end{thm}

Lastly, note that given $d\in [n]$ and any $x\in [q]^n$, there are exactly $\binom {n}{d}(q-1)^{d}$ words $y\in [q]^n$ with $d_H(x,y) = d$. In Section 6, we prove a supersaturated version of Theorem \ref{newcodebound} (which is essentially best possible):

\begin{thm}
 \label{supersaturationcodedistance}
Given $\epsilon, \eta \in (0,1/2)$ there is $\delta '>0$ such that the following holds. Let $\C \subset [q]^n$ with $|\C |> q^{(1-\delta ')n}$ and $d\in {\mathbb N}$ with $\epsilon n< d < (1-\epsilon )n$ (and $d$ even if $q=2$). Then there are at least $\binom {n}{d} (q-1)^d|\C |q^{-\eta n}$ pairs $x,y\in \C $ with $d_H(x,y) = d$.
\end{thm}

\noindent \textbf{Notation:} Given a set $X$, ${\mathcal P}(X)$ will denote the power set of $X$ and $\binom {X}{k}$ will denote the collection of all subsets of size $k$ in $X$. Given $m,n\in {\mathbb N}$ with $m\leq n$, $[n] = \{1,\ldots ,n\}$ and $[m,n] = \{m,\ldots ,n\}$. We also write $(n)_m$ for the falling factorial $(n)_m = n(n-1)\cdots (n-m+1)$.

%
%
%
%
%
%
%
%

\section{Forbidding one intersection}

In this section we give our new proof of Theorem \ref{FR}. We start by recalling the probabilistic technique known as dependent random choice. The reader is directed to the recent survey of Fox and Sudakov \cite{FS} where many other interesting applications of the method are discussed. The following lemma gives a statement of the method which we will use in our applications. We include the short proof for convenience.

\begin{lem}
 \label{deprandchoice}
 Suppose that $G = (X,Y,E)$ is a bipartite graph with $|X|=M, |Y|=N$ and $|E|=\alpha MN$. Then, for any $t\in {\mathbb N}$, there exists $X' \subset X$ with $|X'| \geq \alpha ^tM/2$ with the property that for all $x_1,x_2\in X'$ we have $|N_G(x_1) \cap N_G(x_2)| \geq \alpha M^{-1/t}N$.
\end{lem}

\begin{proof} 
To begin choose uniformly at random $t$ elements $T$ with replacement from $Y$ and let $S$ denote the set of elements adjacent to all elements of $T$. By linearity of expectation
\begin{equation*}
 \mathbb{E}(|S|) = \sum _{x\in X} \Big( \frac {|N_G(x)|}{|Y|} \Big)^t \geq \alpha ^tM
\end{equation*}
where the inequality follows from the convexity of the function $f(z) = z^t$.

We will say that a pair $x,x'$ in $S$ are \emph{bad} if $|N_G(x) \cap N_G(x')| < \alpha M^{-1/t}N$. Now any bad pair has probability at most $(\frac {|N_G(x_1) \cap N_G(x_2)|}{N})^t \leq \alpha ^tM^{-1}$ of appearing in $S$. Therefore, letting $Z$ denote the number of bad pairs in $S$, we find that 
\begin{equation*}
 \mathbb{E}(Z) \leq \alpha ^tM^{-1}\binom {|X|}{2} \leq \alpha ^tM/2.
\end{equation*}
In particular, $\mathbb{E}(|S| - Z) \geq \alpha ^tM/2$. Fix a choice of $T$ such that $|S| - Z$ is at least this big and delete one element from each bad pair $x_1,x_2$ in $S$. Taking $X'$ to be the remaining set, we have $|X'| \geq |S|-Z \geq \alpha ^tM/2$ and no pairs in $X'$ are bad, as required.
\end{proof}

The next lemma shows how one can use Lemma \ref{deprandchoice} to build fixed intersections from smaller ones.

\begin{lem}
 \label{buildingforbiddenintersect}
 For $i = 1,2$, suppose that $n_i, k_i, l_i \in {\mathbb N}$ and $p_i \in (0,1)$ are such that any $l_i$-avoiding family $\A _i \subset \binom {[n_i]}{k_i}$ satisfies $|\A _i| \leq p_i \binom {n_i}{k_i}$.  Suppose that $t \in {\mathbb N}$ satisfies ${\binom {n_1}{k_1}}^{-2} > p_2^t$. Then any $(l_1 + l_2)$-avoiding family ${\mathcal A} \subset \binom {[n_1+n_2]}{k_1+k_2}$ satisfies $|{\mathcal A}| \leq (2p_1)^{1/t} \binom {n_1 + n_2}{k_1 + k_2}$.
\end{lem}

\begin{proof} 
Let ${\mathcal A} \subseteq \binom {[n_1+n_2]}{k_1+k_2}$ be an $(l_1+l_2)$-avoiding family with $|\A | = \alpha \binom {n_1+n_2}{k_1+k_2}$. We wish to show that $\alpha \leq (2p_1)^{1/t}$. To begin, partition $[n_1 + n_2]$ uniformly at random into two sets $V_1$ and $V_2$ of size $n_1$ and $n_2$ respectively. Let $\A ' \subseteq \A $ denote the set 
\begin{equation*}
\A ' = \{A \in \A : |A \cap V_1| = k_1, |A\cap V_2| = k_2\}
\end{equation*}
and let $Z$ denote the random variable $Z = |\A'|$. It is easy to see that $\mathbb{E}(Z) = \alpha \binom{n_1}{k_1}\binom {n_2}{k_2}$. We fix a partition $V_1 \cup V_2 = [n_1 + n_2]$ for which $Z$ is at least this large.

Now we can view $\A '$ as the edge set of a bipartite graph $G = (X,Y,E)$ with vertex bipartition $X = \binom {V_1}{k_1}$ and $Y = \binom {V_2}{k_2}$ in which $AB \in E(G)$ when $A \cup B \in \A '$. We see that $G$ has at least $\alpha |X||Y|$ edges. Apply Lemma \ref{deprandchoice} to $G$ with $t$ as in the statement to find a set $X ' \subset X$ with $|X '| \geq \alpha ^t|X|/2$ such that all distinct pairs $A_1,A_2 \in X'$ have at least $\alpha |X|^{-1/t}|Y|$ common neighbours in $G$. Now if $\alpha > (2p_1)^{1/t}$ then $\alpha ^t/2 > p_1$ and by definition of $p_1$, we find $A_1,A_2 \in X'$ with $|A_1 \cap A_2| = l_1$. 

Let $\B '$ denote the set of common neighbours of $A_1$ and $A_2$ in $G$. By Lemma \ref{deprandchoice} we find that
\begin{equation*}
| {\mathcal B}' | \geq \alpha |X|^{-1/t}|Y| > (2p_1)^{1/t}|X|^{-1/t}|Y| \geq {\binom {n_1}{k_1}}^{-2/t}|Y| > p_2 \binom {n_2}{k_2}.
\end{equation*} 
The third inequality here holds since by definition of $p_1$ we have $p_1\geq 1/\binom {n_1}{k_1}$ and the fourth holds as $\binom {n_1}{k_1}^{-2}> p_2^t$. But now by definition of $p_2$, there exists $B_1,B_2 \in {\mathcal B} '$ with $|B_1 \cap B_2| = l_2$. By construction it can be seen that we have $A_1 \cup B_1, A_2 \cup B_2 \in \A $ and clearly $|(A_1 \cup B_1) \cap (A_2 \cup B_2)| = l_1 + l_2$, as required.
\end{proof}

We will also make use of a theorem of Frankl and Wilson from \cite{FrWil}.

\begin{thm}[Frankl-Wilson]
 \label{FW}
Let $k,l \in {\mathbb N}$ such that $k-l$ is a prime power and $2l + 1 \leq k$. Suppose that ${\mathcal A} \subseteq \binom {[n]}{k}$ is an $l$-avoiding family. Then $|\A | \leq \binom {n}{k - l - 1}$.
\end{thm}

The following simple corollary of Theorem \ref{FW} will give us a slightly more convenient bound.

\begin{cor}
\label{compactFW}
Let $\epsilon \in (0,1)$ and let $l,k\in \N $ with $l<k$, such that $k-l$ is prime with $\max (0,2k-n) + \epsilon n < l < k - \epsilon n$. Then any $l$-avoiding family $\A \subset \binom{[n]}{k}$ satisfies $|\A | \leq c^n \binom {n}{k}$ where $c = c(\epsilon ) < 1$.
\end{cor} 

\begin{proof} 
Let $\A $ be an $l$-avoiding family with $|\A | = \alpha \binom {n}{k}$. By averaging, there exists a set $T \in \binom {[n]}{l - \epsilon n}$ such that $\A _T = \{A \in \binom {[n]\setminus T}{k-|T|}: A \cup T \in \A \}$ has size $|\A _T| \geq \alpha \binom {n - |T|}{k-|T|}$. Setting $l' = \epsilon n$ and $k' = k-|T|$ it is easy to see that $\A _T$ is an $l'$-avoiding $k'$-uniform family. Since $k' = k - l + \epsilon n \geq 2 \epsilon n + 1= 2l' + 1$ and $k'-l' = k-l$ is prime, by Theorem \ref{FW}, we have $|\A _T| \leq \binom {n - |T|}{k'-l'} =  \binom {n - |T|}{k-l}$. This gives that 
\begin{eqnarray*}
 \alpha 
& \leq & \frac { \binom {n- |T|}{k-l} }{\binom {n - |T|}{k - |T|}} 
=
 \frac{(k-|T|)!(n-k)!}{(k-l)!(n-k + \epsilon n)!}\\
& = &
\frac {(k - l + \epsilon n)_{\epsilon n}}
	  {(n - k + \epsilon n)_{\epsilon n}} 
 \leq \big (\frac {k-l + \epsilon n}{n-k + \epsilon n} \big )^{\epsilon n}
 \leq \big (\frac {1}{1 +\epsilon } \big)^{\epsilon n}
\end{eqnarray*}
since $n-k \geq k - l + \epsilon n $. Taking $c = (\frac{1}{1+\epsilon })^{\epsilon } < 1$, the result follows.
\end{proof}

Lastly, we will use the following Vinogradov-type result due to Baker and Harman \cite{BaHa} which says that every large enough odd number can be written as a sum of three primes of almost equal size. 

\begin{thm}[Baker-Harman]
 \label{primepartition}
 Every odd integer $n>n_0$ can be written as a sum of three primes $n = a_1 + a_2 + a_3$ with $|a_i - n/3| \leq n^{4/7}$ for all $i$.
\end{thm}

\begin{proof}[Proof of Theorem \ref{FR}.]
Let $\A \subset \binom {[n]}{k}$ be an $l$-avoiding family which satisfies $l\in [\max (0,2k-n) + \epsilon n, k-\epsilon n]$. We wish to show that $|\A | \leq (1-\delta )^n \binom {n}{k}$, where $\delta = \delta (\alpha , \epsilon )> 0$. By taking $\delta$ to be sufficiently small, we may assume that the theorem holds for small values of $n \leq n_0 = n_0(\epsilon )$, so we will assume that $n\geq n_0$. 

First suppose that $k-l$ is odd. Choose $k_1,k_2,k_3\in {\mathbb N}$ and $n_1,n_2,n_3\in {\mathbb N}$ with $\sum _{i=1}^3 k_i = k$ and $\sum _{i=1}^3 n_i = n$ with $|k_i - k/3| < 1$ and $|n_i - n/3| < 1$ for all $i$, with $n_1\geq n_2 \geq n_3$. By Theorem \ref{primepartition}, as $k-l>\epsilon n$ and $n> n_0(\epsilon )$, we can write $k-l = a_1 + a_2 + a_3$ where $a_i$ is prime and $|(k-l)/3 - a_i| \leq \frac {\epsilon n}{8}$ for all $i$. Also set $l_i = k_i - a_i$ for all $i$. Then $k_i - l_i$ is prime for all $i$, $\sum _{i} k_i-l_i = k-l$ and $\max (0,2k_i - n_i) +\epsilon n_i/2 \leq l_i \leq k_i - \epsilon n_i /2$.

By Corollary \ref{compactFW} any $l_i$-avoiding family $\A _i \subset \binom {[n_i]}{k_i}$ satisfies $|\A _i| \leq p_i \binom {n_i}{k_i}$ where $p_i = c_1 ^{n_i}$ with $c_1= c(\epsilon /2)<1$. Taking $t_1 = \lceil 2/ \log _2(1/c_1) \rceil $ we find
\begin{equation*}
p_2^{t_1} = c_1^{t_1 n_2} \leq 2^{-2n_1} < \binom {n_1}{k_1}^{-2}.
\end{equation*} 
Therefore, by Lemma \ref{buildingforbiddenintersect} any $(l_1 + l_2)$-avoiding family $\B \subset \binom {[n_1 + n_2]}{k_1 + k_2}$ with $|\B | =\beta \binom {n_1 + n_2}{k_1 + k_2}$ satisfies $\beta \leq (2c_1^{n_1})^{1/t_1}$.

To complete the proof we simply repeat the previous argument again. Let $t_2 = \lceil 4t_1/\log _2 (1/c_1) \rceil$. Then we have
\begin{equation*}
\beta ^{t_2} \leq ((2c_1^{n_1})^{1/t_1})^{t_2} \leq (2c_1^{n_1})^{4/\log _2(1/c_1)} \leq 2^{-2n_3} < \binom {n_3}{k_3}^{-2}
\end{equation*} 
where the third inequality holds since $n\geq n_0(\epsilon )$. Lemma \ref{buildingforbiddenintersect} now gives that any $l$-avoiding family $\A \subset \binom {[n]}{k}$ satisfies $|\A | \leq c_2^n \binom {n}{k}$ where $c_2 = (c_1^{n_3}) ^{1/t_2n} \leq c_1^{1/4t_2} <1$. As $c_1$ and $t_2$ depend only on $\epsilon $, this completes the proof in the case when $k-t$ is odd.

The case where $k-t$ is even can be proved by splitting $k-t$ into $4$ primes of almost equal size. The proof now proceeds identically to the odd case, using an additional application of Lemma \ref{buildingforbiddenintersect}. 
\end{proof}
 
%
%
%
%
%
%
%
%
 
\section{Forbidding code distances}

In this section we prove Theorem \ref{newcodebound}. We will assume that $q \geq 3$ throughout the section, as the case $q = 2$ follows from Theorem \ref{FR}. We require the following definition:

\begin{dfn} Given a prime $p$ and a set $\D \subset \Z _p \setminus \{0\}$, we say that a code ${\mathcal C} \subset [q]^n$ is a $\D \pmod p$-code if for all $d \in d({\mathcal C})$, we have $d \equiv d'\mod p$ for some $d' \in {\mathcal D}$.
\end{dfn}

The following theorem, due to Frankl \cite{FrCode} (see also \cite{BSW}), gives an upper bound on the size of$\pmod p$-codes .

\begin{thm}[Frankl]
 \label{HammingDistanceBoundfor[q]}
Suppose that $p$ is a prime and that ${\mathcal C} \subset [q]^n$ is a $\D \pmod p$-code with $|\D | = l$. Then $|{\mathcal C}| \leq \sum _{i = 0}^{l} \binom {n}{i} (q-1)^{i} $.
\end{thm}

In applying Theorem \ref{HammingDistanceBoundfor[q]} we use the following estimate due to Chernoff \cite{chernoff}. Let $q \in {\mathbb N}$ with $q\geq 3$. Then given $\alpha \in (0,(q-1)/q)$, we have 
\begin{equation*}
 S_q(\alpha ,n) := \sum _{i=0}^{\alpha n} \binom {n}{i} (q-1)^i \leq q^{f_q(\alpha )n}
\end{equation*}
where $f_q(\alpha ) = \alpha \log _q(\frac {q-1}{\alpha }) + (1-\alpha )\log _q(\frac {1}{1-\alpha })$. 

\begin{prop}
 \label{ChernoffTailEstimate}
 Given $q\geq 3$ and $\alpha \in [0,3/5]$ we have $S_q(\alpha , n) \leq q^{(1 - 1/125)n}$.
\end{prop}

\begin{proof}
First note the following:
\begin{enumerate}[(i)]
 \item \label{increasing in alpha}
 $\frac{\partial f_q}{\partial \alpha }(\alpha) = \log _q \Big [ \frac{(q-1)(1-\alpha )}{\alpha} \Big ] \geq 0$ for $\alpha \in [0,(q-1)/q]$;
 \item \label{convex in alpha}
 $\frac{\partial ^2 f_q}{\partial \alpha ^2}(\alpha) = \frac{1}{\log _e q} \Big [ - \frac{1}{1-\alpha } - \frac{1}{\alpha} \Big ] \leq 0$, so $f_q(\alpha )$ is concave as a function of $\alpha $ on $[0,1]$. As $f_q(0) = 0$ and $f_q(\frac{q-1}{q}) = 1$, this shows that $f_q(\alpha ) \geq \frac{q\alpha }{q-1}$ for $\alpha \in [0,(q-1)/q]$;
 \item \label{decreasing in q}
 $\frac{\partial f_q}{\partial q}(\alpha) = \frac{1}{q \log _e q} \Big [\frac{q\alpha }{q-1} - f_q(\alpha ) \Big ]\leq 0$ for $\alpha \in [0,(q-1)/q]$ by (\ref{convex in alpha}). 
\end{enumerate}
But then, for $q\geq 3$ and $\alpha \in [0,3/5] \subset [0,(q-1)/q]$, we have
\begin{equation*}
f_q(\alpha ) \leq f_3(\alpha ) \leq f_3(3/5) \leq 0.992,
\end{equation*} 
where the first inequality holds since $f_q(\alpha)$ is decreasing in $q$ by (\ref{decreasing in q}), the second since $f_3(\alpha)$ is increasing in $\alpha$ by (\ref{increasing in alpha}) and the third by a numerical calculation.
\end{proof}

Combined with Proposition \ref{ChernoffTailEstimate}, Theorem \ref{HammingDistanceBoundfor[q]} now gives the following corollary.

\begin{cor}
 \label{HammingDistanceBoundfor[q]Compact}
Let $\epsilon \in (0,1)$ and $q\geq 3$. Suppose that $p$ is a prime with $\epsilon n < p < 3n/5$ and that ${\mathcal C} \subset [q]^n$ is a code with $p\notin d({\mathcal C})$. Then $|{\mathcal C}| \leq q^{(1-\delta _1)n}$ where $\delta _1 = \delta _1(\epsilon )> 0$.
\end{cor}

\begin{proof} Suppose that $|{\mathcal C}| = \alpha q^n$. Choose $t$ so that $p \in (\frac{n-t}{2}, \frac{3(n-t)}{5})$ -- this is possible by the stated bound on $p$ above. Now given a set $T \in \binom {[n]}{t}$ and elements $a_i \in [q]$ for $i\in T$, let 
\begin{equation*}
{\mathcal C} _T = \{x \in {\mathcal C}: x_i = a_i \mbox { for all } i\in T\}.
\end{equation*} 
By averaging we find $T \in \binom {[n]}{t}$ and $\{a_i\in [q]: i\in T\}$ such that $|{\mathcal C} _T| \geq \alpha q^{n-t}$. View ${\mathcal C}_T$ as a subset of $[q]^{[n]\setminus T}$. Clearly $p\notin d({\mathcal C}_T)$. Since $p > (n-t)/2$, the set ${\mathcal C} _T$ is a ${\mathcal D} \pmod p$ code in $[q]^{[n]\setminus T}$, where ${\mathcal D} = \{1,\ldots ,p-1\}$. Therefore  by Theorem \ref{HammingDistanceBoundfor[q]} and Proposition \ref{ChernoffTailEstimate}
\begin{equation*}
\alpha q^{n-t} \leq |{\mathcal C}_T| \leq S_q(3/5 ,n-t) \leq q^{(1-1/125)(n-t)}.
\end{equation*}
Therefore $\alpha \leq q^{-(n-t)/125} \leq q^{- \epsilon n/125}$ using that $\epsilon n \leq  p \leq n-t$. Taking $\delta _1(\epsilon ) = \epsilon /125$ completes the proof.
\end{proof}

Corollary \ref{HammingDistanceBoundfor[q]Compact} will allow us to deal with forbidden distances which are not too large. For larger distances we will use the following diametric theorem for $[q]^n$ due to Ahlswede and Khachatrian \cite{AK}. The diameter of a set ${\mathcal C} \subset [q]^n$, $\diam ({\mathcal C})$, is defined as $\diam ({\mathcal C}) = \max \{d: d\in d({\mathcal C})\}$. For $r\in {\mathbb N} \cup \{0\}$ let $\mathcal K _r \subset [q]^n$ denote the set 
\begin{equation*}
 {\mathcal K _r} = \{ v \in [q]^n: |\{i\in [t + 2r]: v_i = 1\}| \geq t + r\}.  
\end{equation*}
It is easy to see that $\diam ({\mathcal K} _r) = n-t$ for all $r$.

\begin{thm}[Ahlswede, Khachatrian]
 \label{diametrictheorem}
 Let $q,t \in {\mathbb N}$ with $q \geq 2$ and let $r \in  {\mathbb N} \cup \{0\}$ be the largest integer such that 
 \begin{equation}
   \label{controlfordiametric}
  t + 2r < \min \Big \{ n + 1, t + 2\frac{t-1}{q-2} \Big \}.
 \end{equation}
Then any code ${\mathcal C} \subset [q]^n$ with $\diam ({\mathcal C}) \leq n-t$ satisfies $|{\mathcal C}| \leq |{\mathcal K}_r|$. (\textit{By convention, $(t-1)/(q-2) = \infty $ if q = 2.})
\end{thm}

We will use the following simple consequence of Theorem \ref{diametrictheorem}.

\begin{cor}
 \label{diametrictheoremcompact}
 Given $\epsilon \in (0,1/3)$ and $q \in {\mathbb N}$ with $q\geq 3$, every set ${\mathcal C} \subset [q]^n$ with $\diam ({\mathcal C}) \leq  (1 -\epsilon )n$ satisfies $|{\mathcal C}| \leq q^{(1-\delta _2)n}$ where $\delta _2 = \delta _2(\epsilon )>0$.
\end{cor}
 
\begin{proof}
Let $t = \epsilon n$. Since $\epsilon < 1/3$, we have 
\begin{equation*}
n+1 > \epsilon n + 2 \epsilon n - 1 \geq t + 2\frac{t-1}{q-2},
\end{equation*}
so the minimum in (\ref{controlfordiametric}) is attained by the right hand term and gives $r = \lceil (t-1)/(q-2) \rceil -1$ in Theorem \ref{diametrictheorem}. Therefore to prove the statement, by Theorem \ref{diametrictheorem} it suffices to prove that $|{\mathcal K}_r| \leq q^{(1-\delta _2)n}$. We have
\begin{eqnarray*}
 |{\mathcal K}_r| & = & \Big ( \sum _{i=0}^{r} (q-1)^{i} \binom {t + 2r}{i} \Big ) q^{n-t -2r} = S_q\big (\frac {r}{t+2r},t+2r\big ) q^{n-t-2r}\\
 & \leq & q^{(1-1/125)(t+2r)}q^{n - t - 2r} = q^{n-(t+2r)/125} \leq q^{(1 - \epsilon /125)n}, 
\end{eqnarray*}
using Proposition \ref{ChernoffTailEstimate} in the first inequality and that $\epsilon n \leq t < t+2r$ in the second. Taking $\delta _2(\epsilon ) = \epsilon /125$ completes the proof.
\end{proof}

The next lemma is an analogue of Lemma \ref{buildingforbiddenintersect} for subsets of $[q]^n$ and can be proved similarly.

\begin{lem}
 \label{buildingcodedistance}
  For $i = 1,2$, suppose that $n_i, d_i \in {\mathbb N}$ and $p_i \in (0,1)$ are such that if $\mathcal C _i \subset [q]^{n_i}$ with $d_i \notin d({\mathcal C}_i) $ then $|{\mathcal C} _i| \leq p_i q^{n_i}$. Suppose that $t \in {\mathbb N}$ satisfies $q^{-2n_1} > p_2^t$. Then any set ${\mathcal C} \subset [q]^{n_1+n_2}$ with $d_1 + d_2 \notin d({\mathcal C})$ satisfies $|{\mathcal C}| \leq (2p_1)^{1/t} q^{n_1+n_2}$.
\end{lem}

We are now ready for the proof of Theorem \ref{newcodebound}. Given a partition $[n] = V_1 \cup \cdots \cup V_k$ of $[n]$ and vectors $x_l \in [q]^{V_l}$ for all $l\in [k]$, we will write $x_1\circ \cdots \circ x_k \in [q]^n$ for the concatenation of $x_1,\ldots ,x_k$, where
\begin{equation*}
 (x_1\circ \cdots \circ x_k)_i = (x_l)_i \mbox { if }i\in V_l.
\end{equation*}

\begin{proof}[Proof of Theorem \ref{newcodebound}] 
Let ${\mathcal C} \subset [q]^n$ with $|{\mathcal C}| = \alpha q^n$ where $q\geq 3$ and suppose that for some $d\in [\epsilon n,(1-\epsilon )n]$ we have $d\notin d(\C )$. We wish to show that $\alpha \leq (1-\delta )^n$ where $\delta = \delta (\epsilon )>0$. By taking $\delta $ sufficiently small, we can assume that the result holds for $n < n_0(\epsilon)$, so we will assume that $n\geq n_0$. The proof will split into two pieces, according as $d \in [\epsilon n , \frac{11}{20} n]$ or $d\in [\frac{11}{20} n, (1-\epsilon )n]$.\\

\noindent \textbf{Case 1:  $d \in [\epsilon n,\frac{11}{20} n]$}

We will suppose that $d$ is odd, as the case of even $d$ is similar. As $d \geq \epsilon n$, for $n \geq n_0(\epsilon )$, Theorem \ref{primepartition} allows us to write $d = d_1 + d_2 + d_3$ where $d_i$ are primes with $|d_i - d/3| \leq \epsilon n/100$. We also partition $n$ as a sum of naturals $n = n_1 + n_2 + n_3$ where $|n_i - n/3| \leq 1$ with $n_1 \geq n_2 \geq n_3$. For $n\geq n_0(\epsilon )$ this gives for all $i\in [3]$ that
\begin{equation*}
\epsilon n_i /2 \leq d/3 - \epsilon n/100 \leq d_i \leq d/3 + \epsilon n/100 \leq 3n_i/5.
\end{equation*} 

Set $V_1 = [n_1]$, $V_2 = [n_1 +1,n_1 + n_2]$ and $V_3 = [n_1 + n_2 +1, n]$. By Corollary \ref{HammingDistanceBoundfor[q]Compact} any code ${\mathcal C} _i \subset [q]^{V_i}$ with $d_i \notin d({\mathcal C _i})$ satisfies $|\mathcal C _i| \leq p_i q^{n_i}$ where $p_i = q^{-\delta _1(\epsilon /2)n_i}$. Taking $t_1 = \lceil 4/ \delta _1(\epsilon /2) \rceil $, we find that
\begin{equation*}
p_2^{t_1} = q^{-t_1 \delta _1(\epsilon /2) n_2} \leq q^{-2n_1},
\end{equation*} 
since $n_2 \geq n_1/2$. Therefore, by Lemma \ref{buildingcodedistance} any code $\B \subset [q]^{V_1 \cup V_2}$ with $d_1 + d_2 \notin d(\B )$ satisfies $|{\mathcal B}| \leq \alpha _1 q^{n_1 + n_2}$ where $\alpha _1 = (2q^{-\delta _1(\epsilon /2)n_1})^{1/t_1}$. 

To complete the proof we repeat the previous argument. Let $t_2 = \lceil 4t_1/\delta _1(\epsilon /2)\rceil $. Then we find that 
\begin{equation*}
\alpha _1^{t_2} = (2q^{-\delta _1(\epsilon /2)n_1})^{t_2/t_1} \leq (2q^{-\delta _1(\epsilon /2)n_1})^{4/\delta _1(\epsilon /2)}= 2^{4/\delta _1(\epsilon /2)}q^{-4n_1} \leq q^{-2n_3}.
\end{equation*}
The last inequality holds since $n\geq n_0(\epsilon )$. Letting $\delta _3(\epsilon ) = \delta _1(\epsilon /2)/4t_2$, Lemma \ref{buildingcodedistance} now gives that any family ${\mathcal C} \subset [q]^n$ with $d = d_1+d_2+d_3 \notin d({\mathcal C})$ satisfies $|{\mathcal C}| \leq \alpha _2^n q^n$ where $\alpha _2 = (2q^{-\delta _1 (\epsilon /2)n_3}) ^{1/t_2n} \leq q^{-\delta _1(\epsilon /2)/4t_2} = q^{-\delta _3 (\epsilon )}$ for $n> n_0$.\\

\noindent \textbf {Case 2: $d \in [\frac{11}{20} n,(1-\epsilon )n]$} 

We will prove this case using the previous one. Let $\delta _3$ be as in Case 1 above and let $\delta _2$ denote the function in Corollary \ref{diametrictheoremcompact}. First, choose $n_1 \in [n]$ such that $|\frac{29}{40}n_1 + \frac{11}{40} n - d| \leq 1$. As $d \geq \frac{11}{20}n$, this gives $n/4 \leq n_1 \leq (1-\epsilon )n$. Take $t = \lceil 2/\epsilon \delta _3(1/4)\rceil$ and $\delta _4(\epsilon ) = \delta _2(\epsilon /4)/8t>0$. We will show that if $\C \subset [q]^n$ where $|\C | = \alpha q^n$ with $\alpha > q^{-\delta _4(\epsilon )n}$ then $\C $ contains two words at Hamming distance $d$.

To begin, partition $[n] = V_1 \cup V_2$ where $V_1 = [n_1]$ and $V_2 = [n_1+1,n]$. We set $n_2 = n - n_1 = |V_2|$. As in Lemma \ref{buildingforbiddenintersect}, view the elements of ${\mathcal C}$ as edges of a bipartite graph $G = (X,Y,E)$ with bipartition $X = [q]^{V_1}$ and $Y = [q]^{V_2}$, where $xy \in E(G)$ if $x\circ y \in {\mathcal C}$. Clearly $|E(G)| = \alpha |X||Y|$.
Apply Lemma \ref{deprandchoice} to $G$ with $t$ as above to find a set $X' \subset X $ with 
\begin{equation*}
|X'| = \alpha ^t|X|/2 > q^{-\delta _2(\epsilon /4) n/8}q^{n_1}/2 \geq q^{(1-\delta _2(\epsilon /4))n_1}
\end{equation*} 
(using that $n_1 \geq n/4$ and $n\geq n_0(\epsilon )$) such that all distinct $x,x'$ in $X'$ share at least $\alpha |X|^{-1/t}|Y|$ common neighbours in $Y$. By Corollary \ref{diametrictheoremcompact} there exists $x,x' \in X'$ with $d' = d_H(x,x')$ satisfying $(1-\epsilon /4)n_1 \leq d' \leq n_1$. Let ${\mathcal B} \subset [q]^{V_2}$ denote the set of common extensions of $x,x'$ in $Y$. We have 
\begin{eqnarray*}
\log _q |\mathcal B| & \geq & \log _q(\alpha |X|^{-1/t}|Y|) > \log _q (q^{-\delta _4(\epsilon )n -n_1/t}q^{n_2})\\
 & > & -2n_1/t + n_2 \geq -2n_1\epsilon \delta _3(1/4)/2 + n_2\\
 & = & -\epsilon \delta _3(1/4)n_1 + n_2  =  - \delta _3(1/4)n_2 \Big (\frac{\epsilon n_1}{n_2}\Big ) + n_2\\
 & \geq & (1-\delta _3(1/4))n_2.
\end{eqnarray*}
Here we used that $n_1/n_2 = n_1/(n-n_1) \leq 1/\epsilon $ since $n_1 \leq (1-\epsilon )n$. Therefore $|{\mathcal B}| > q^{(1- \delta _3(1/4))n_2}$.
As $d' \in [(1-\epsilon /4)n_1,n_1]$ and $|(d - n_1) - \frac{11}{40} n_2 | \leq 1$ by choice of $n_1$, we have
\begin{eqnarray*}
d - d' & \in & [d - n_1, d - (1-\epsilon /4)n_1]\\ 
& \subset & [\frac{11}{40}(n-n_1) -1, \frac{11}{40} (n - n_1) + \epsilon n_1 /4]\\
& \subset & [\frac{1}{4} n_2, \frac{11}{20} n_2].
\end{eqnarray*} 
Here we again used that $n_1/(n-n_1) \leq 1/\epsilon $. Therefore, by definition of $\delta _3$, as ${\B }\subset [q]^{V_2}$ and $|{\mathcal B}| > q^{(1-\delta _3(1/4))n_2}$, there exists a pair $y,y' \in {\mathcal B}$ with $d_H(y,y') = d - d'$. But this gives $x\circ y, x'\circ y' \in {\mathcal C}$ and $d_H(x\circ y,x'\circ y') = d$, completing the proof of this case.\\

Taking $\delta (\epsilon ) = \min (\delta _3(\epsilon ), \delta _4 (\epsilon ))$ completes the proof of the Theorem.\end{proof}

%
%
%
%
%
%
%
%

\section{Weak sunflowers in $[q]^n$}

In this section, we will prove Theorem \ref{weaksunflowersfor[q]}. For convenience, we will assume that $n$ is a multiple of $k$ with $n = km$; this assumption can easily be removed. Set $V_i = [(i-1)m+1, im]$ for all $i\in [k]$. We will prove by induction on $k$ that given $\epsilon >0$ and $d \in [\epsilon m ,(1 - \epsilon )m]$ (with $d$ even if $q = 2$), there exists $\delta' = \delta '(\epsilon , k)>0$ with the following property: for any set ${\mathcal C} \subset [q]^n$ with $|{\mathcal C}| > q^{(1-\delta ')n}$, there exists $x_i, y_i \in [q]^{V_i}$ for $i\in [k]$ with $d_H(x_i,y_i) = d$, such that $z_1\circ \cdots \circ z_k \in {\mathcal C}$ for any choice of $z_i\in \{x_i,y_i\}$. This will complete the proof  as taking
\begin{equation*}
	v_i = x_1 \circ \cdots \circ x_{i-1} \circ y_i \circ x_{i+1} \circ \cdots \circ x_k,
\end{equation*}
the set $\{v_1,\ldots ,v_k\}$ is a weak-sunflower with $k$ petals contained in $\C$.

The case when $k=1$ follows immediately from Theorem \ref{newcodebound}, so we will assume by induction that the result holds for $k-1$ and prove it for $k$. Let $W_1 = \cup _{i = 2}^k V_i$ so that $[n] = V_1 \cup W_1$. Letting $t = \lceil 2/((k-1)\delta ' (\epsilon , k-1)) \rceil $, we claim that we can take $\delta ' = \delta '(\epsilon ,k) = \delta (\epsilon )/2kt $, where $\delta (\epsilon )$ is as in Theorem \ref{newcodebound}. Indeed, as in the proof of Lemma \ref{buildingforbiddenintersect}, view elements of $[q]^n$ as edges of a bipartite graph $G = ([q]^{V_1},[q]^{W_1}, E)$ in which $xy \in E(G)$ if $x\circ y \in \mathcal C$. Then if $|{\mathcal C}| = |E(G)| =  \alpha q^{n}$ where $\alpha > q^{-\delta ' n}$, by Lemma \ref{deprandchoice}, there exists a set $\C _1 \subset [q]^{V_1}$ with 
\begin{equation*}
|{\mathcal C _1}| \geq \alpha ^tq^m /2 \geq q^{-\delta (\epsilon )n/2k}q^{m}/2 = q^{(1 - \delta (\epsilon )/2)m}/2 \geq q^{(1-\delta (\epsilon ))m}
\end{equation*} 
with any two elements in $\C _1$ sharing at least 
\begin{equation*}
\alpha q^{-m/t}q^{(k-1)m} > q^{-\frac {\delta (\epsilon )n}{2kt} - \frac{m}{t}}q^{(k-1)m} \geq q^{-\frac{2m}{t}}q^{(k-1)m}
\geq 
q^{(1-\delta (\epsilon ,k-1))(k-1)m}
\end{equation*} 
common neighbours in $G$. But then, by Theorem \ref{newcodebound}, $\C _1$ must contain elements $x_1$ and $x_2$ with $d_H(x_1,x_2) = d$. Also, by the induction hypothesis for $k-1$, we find $x_i,y_i \in [q]^{V_i}$ for all $i\in [2,k]$ with $d_H(x_i,y_i) = d$ such that all elements of the set 
\begin{equation*}
\{ z_2\circ \cdots \circ z_k: z_i \in \{x_i,y_i\} \mbox { for all } i\in [2,k]\}
\end{equation*} 
are common neighbours of both $x_1$ and $y_1$. But by definition of $G$, this means that $z_1\circ \cdots \circ z_l \in \C $ for any choice of $z_i \in \{x_i,y_i\}$ for all $i\in [k]$, as claimed.

%
%
%
%
%
%
%
%

\section{Forbidding distances between pairs of sets in $[q]^n$}

In this section we prove Theorem \ref{crosscodedistance}. Given $\epsilon $ we will take $\delta '(\epsilon ) = \delta (\epsilon /2)/2$, where $\delta (\epsilon /2)$ is as in Theorem \ref{newcodebound} and $\gamma = \min (\epsilon /2, \frac{\delta (\epsilon/2)}{16\log (1/\delta (\epsilon/2))})$. Let $q \geq 3$ and suppose that $\C ,\D \subset [q]^n$ with $|\C | \geq q^{(1- \delta ')n}$ and such that for all $x\in \C$ there is $y\in \D $ with $d_H(x,y) \leq \gamma n$. Suppose $d\in (\epsilon n,(1-\epsilon )n)$. We will show that there exists $x \in \C $ and $y\in \D $ with $d_H(x,y) = d$.

From the statement, for all $x\in {\mathcal C}$ there is some $y_x \in {\mathcal D}$ with $d_H(x,y_x) \leq \gamma n$. By pigeonholing, there must be a set $T \subset \binom {[n]}{\gamma n}$ and a subset ${\mathcal C}' \subset {\mathcal C}$ with $|{\mathcal C}'| \geq |{\mathcal C}|/\binom {n}{\gamma n} \geq |{\mathcal C}|2^{-H(\gamma )n}$ with the property that, for all $x\in {\mathcal C}'$, we have $\{i\in [n]: (x)_i \neq (y_x)_i\} \subset T$. There are at most $q^{\gamma n}$ choices for both 
$x|_T$ and $y_x|_T$, so again by pigeonholing we find ${\mathcal C}'' \subset {\mathcal C}'$ with $|{\mathcal C}''| \geq |{\mathcal C}'|/q^{2\gamma n}$ and vectors $f_0,g_0 \in [q]^{T}$ such that $x|_T = f_0$ and $y_x|_{T} = g_0$ for 
all $x \in {\mathcal C}''$. Let $d_H(f_0,g_0) = t \leq \gamma n \leq \epsilon n/2$. Now by choice of $\gamma $, we have $H(\gamma ) \leq \delta (\epsilon /2)/4$ and $\gamma < \delta (\epsilon /2) /8$ and so
\begin{equation*}
|{\mathcal C}''| \geq |{\mathcal C}|2^{-H(\gamma )n}q^{-2\gamma n} \geq q^{(1- \delta (\epsilon/2))n}.
\end{equation*} 
Therefore, since $\epsilon n/2 \leq d - \gamma n \leq d-t \leq (1 -\epsilon )n$ by Theorem \ref{newcodebound} there are $x,x' \in {\mathcal C}''$ with $d_H(x,x') = d-t$. But then 
\begin{equation*}
 	d_H(x',y_x) = \underbrace{d_H(x,y_x)}_\text{distance in $T$} + \underbrace{d_H(x',x)}_\text{distance in $[n]\setminus T$} = d_H(f_0,g_0) + d_H(x',x) = d.
\end{equation*}
As $x' \in {\mathcal C}$ and $y_x \in {\mathcal D}$, this completes the proof. \hspace{4cm} $\square $
\vspace{.2mm}

We note that after passing to a large subset, the conditions of Theorem \ref{crosscodedistance} are easily seen to hold when $|\C |, |\D | \geq (q-\delta )^n$ if $\delta = \delta (\epsilon ,q ) > 0$ is sufficiently small. Indeed, given $\epsilon >0$, let $\gamma $ be as in Theorem \ref{crosscodedistance}. Let $\C ' \subset \C $ denote the set 
\begin{equation*}
\C '= \{ x \in \C : \not \exists y\in \D \mbox{ with }d_H(x,y) \leq \gamma n\}.
\end{equation*} 
We claim that $|\C '| \leq |\C |/2$. Indeed, otherwise writing $N^{(t)}(\C ') = \{x\in [q]^n: d_H(x,x') \leq t \mbox{ for some } x' \in \C '\}$, we have 
\begin{equation}
 \label{neighbourhoodsize}
\D \subset [q]^n \setminus N^{(\gamma n)}(\C ').
\end{equation} 
But for $\delta = \delta (\gamma ,q) >0$ small enough, since $|\C' | \geq (q-\delta )^n/2$, by an approximate vertex isoperimetric inequality for $K_q^n$ (see \cite{Harper} or \cite{BolLead}), we find that $|N^{(\gamma n)}(\C ')| > q^n - (q-\delta )^n$. But by (\ref{neighbourhoodsize}), this contradicts $|{\mathcal D}| \geq (q-\delta )^n$.

%
%
%
%
%
%
%
%

\section{Supersaturated version of Theorem \ref{newcodebound}}

In this section we prove Theorem \ref{supersaturationcodedistance}. To begin, set $\alpha = \eta / (16\log (16/\eta ))$ and $\delta ' = \eta \epsilon \delta (\alpha /2 )/8$ where $\delta $ is as in Theorem \ref{newcodebound}. Also $m = \alpha n$ and $r = \max \{\lfloor q^{\eta/4 }\rfloor ,2\}$. Let $\C \subset [q]^n$ with $|\C | > q^{(1-\delta ')n}$. We will show that given $d$ with $\epsilon n \leq d \leq (1-\epsilon )n$, the code $\C $ contains at least $\binom {n}{d}(q-1)^d|\C |q^{-\eta n}$ pairs $x,y\in \C $ with $d_H(x,y) = d$. We start by giving the proof in the case where $d$ is even.

Let $N$ denote the number of pairs $\{x,y\}$ with $x,y \in \C $ such that $d_{H}(x,y) = d$. Make the following selection of random choices: 
 \begin{itemize}
 	\item choose a partition of $[n] = V_1 \cup V_2$ with $|V_1| = d + m $ and $|V_2| = n - d - m $ uniformly at random;
 	\item for each $i\in V_1$, choose a subset $Q_i \subset [q]$ of size $r$ uniformly at random; 
	\item for each $i\in V_2$, choose an element $q_i \in [q]$ uniformly at random.
 \end{itemize}
We will say that an element $x\in \C $ is a \emph{captured element} if $x_i \in Q_i$ for all $i\in V_1$ and $x_j = q_j$ for all $j\in V_2$. Let ${\mathcal E} \subset \C $ denote the set of captured elements. We also say that a pair $\{x,y\} \in {\mathcal C}^{(2)}$ is a \emph{captured $d$-pair} if $x,y\in {\mathcal E}$, $d_H(x,y) = d$ and $V_2 \subset \mbox {Agree}(x,y) $. 

Let $X$ and $Y$ denote the random variables which count the number of captured elements and the number of captured $d$-pairs respectively.
 Clearly, given $x\in {\mathcal C}$, we have ${\mathbb P}(x\in {\mathcal E}) = r^{d+m}/q^n$. Therefore we have
 \begin{equation}
   \label{capturedelementssize}
  {\mathbb E}(X) = \frac {r^{d+m}|\C |}{q^{n}}.
 \end{equation}
For a fixed pair $x,y\in \C $ with $d_H(x,y)=d$ we have
 \begin{equation*}
  {\mathbb P}(x,y \mbox{ form a captured $d$-pair}) = 
  	\frac{\binom {d+m}{d}}{\binom {n}{d}} \Bigg (\frac {\binom {r}{2}}{\binom{q}{2}} \Bigg )^{d} \Big (\frac {r}{q} \Big )^{m} q^{-(n-d-m)}
 \end{equation*}
To see this we justify one term at a time. The first term is the probability that the $d$ coordinates on which $x$ and $y$ differ are included in $V_1$. The second term is the probability that that the entries $x_i$ and $y_i$ are included in $Q_i$ where they differ. The third term is the probability that the remaining (common) entries of $x_i = y_i$ in $V_1$ are included in $Q_i$ and the last term is the probability that $q_i = x_i$ for $i\in V_2$. 

Suppose for contradiction that $N < \binom {n}{d}(q-1)^d|\C |q^{-\eta n}$. Then
 \begin{eqnarray*}
 	\mathbb {E}(Y) 
 	& = &
 	\Big (\frac {r}{q} \Big )^{m}\Big (\frac {r(r-1)}{q(q-1)}\Big )^{d}q^{-(n-d-m)}\frac{\binom {d+m}{m}}{\binom {n}{d}}N\\
 	& < &
 	\binom {d+m}{m}\frac {r^{d+m} (r-1)^{d}}{q^{n}}|\C | q^{-\eta n}\\
 	 	& \leq & \Big ( \frac {\binom {n}{m}(r-1)^n}{q^{\eta n}} \Big) \frac {r^{d+m}|\C |}{q^n}.
 \end{eqnarray*} 
Now for $\alpha \leq \eta /(16 \log (16/\eta ))$ we have $H(\alpha ) \leq \eta /4$ and so $\binom {n}{m} \leq 2^{H(\alpha )n}< q^{\eta n/4}$. Since we also have $(r-1) \leq \max \{q^{\eta /4},1\} = q^{\eta /4}$ we have
\begin{equation*}
  {\mathbb E}(Y) 
  	\leq 
  		\frac{1}{q^{\eta n/2}} \frac {r^{d+m}|\C |}{q^n} 
  	\leq 
  		\frac {r^{d+m}|\C |}{2q^n}
\end{equation*}
for $n\geq n_0 \geq 2/\eta $. Combined with (\ref{capturedelementssize}), as $|{\mathcal C}| \geq q^{(1-\delta ')n}$, this gives
\begin{equation}
 {\mathbb {E}}(X-Y) \geq \frac {r^{d+m}|\C |}{2q^n} \geq \frac {r^{d+m}q^{-\delta 'n}}{2}.
\end{equation}
But $q^{-\delta 'n} = (q^{\eta /4 })^{-\delta (\alpha /2) \epsilon n/2} \geq r^{-\delta (\alpha /2)\epsilon n/2} \geq r^{-\delta (\alpha /2)(d+m)/2}$ as $d\geq \epsilon n$. This shows that 
\begin{equation}
 \label{differenceCapturedElementsCapturedPairs}
		{\mathbb{E}}(X-Y) 
 	\geq 
 		\frac {r^{(1-\delta (\alpha /2)/2)(d+m)}}{2} 
 	> 
 		r^{(1-\delta (\alpha /2))(d+m)}.
\end{equation}
The second inequality here holds since $r^{\delta (\alpha /2)(d+m)/2} \geq r^{\delta (\alpha /2)\epsilon n/2} \geq 2$ for $n \geq n_0$. Fix choices of $V_1, V_2, Q_i$ for all $i\in V_1$ and $q_i$ for $i\in V_2$ such that $X-Y$ is at least this big. Now remove one element from every captured $d$-pair in ${\mathcal E}$. By (\ref{differenceCapturedElementsCapturedPairs}), this leaves a set ${\mathcal E}' \subset {\mathcal E}$ with 
\begin{equation}
 \label{size_E'}
|{\mathcal E}'| >r^{(1-\delta (\alpha /2))(d+m)}
\end{equation} 
which contains no $d$-pairs. 

Now ${\mathcal E}'$ is a subset of $\prod _{i\in V_1}Q_i\times \prod _{j \in V_2} \{q_j\}$ and this product set is naturally identified with $[r]^{d+m}$. We also have 
\begin{equation*}
\frac{\alpha }{2}(d + m) \leq d \leq \big (1-\frac {\alpha }{2} \big) (d+m).
\end{equation*}
Indeed, $\alpha (d+m)/2 \leq d$ since $m \leq d$ and $\alpha \leq 1$ and $d \leq (1 - \alpha /2)(d+m)$ since $\alpha d/2 \leq \alpha n/2 = m/2 \leq (1-\alpha /2)m$. But now since  ${\mathcal E}'$ does not contain a pair $(x,y)$ with $d_H(x,y) = d$, by Theorem \ref{newcodebound} we have $|{\mathcal E}'| \leq q^{(1-\delta (\alpha /2))(d+m)}$. However this contradicts (\ref{size_E'}). Therefore we must have $N \geq \binom {n}{d}(q-1)^d|\C |q^{-\eta n}$, as required. 

This completes the proof of the case of $d$ even. The case of $d$ odd can be deduced from the even case by dependent random choice. The idea is to partition $V$ as $V_1 \cup V_2$ where $V_1$ is small, find $x,y \in [q]^{V_1}$ at odd distance with many common extensions in $[q]^{V_2}$, then apply the even case to these common extensions. The details are similar to those in previous arguments, so we omit them.

%
%
%
%
%
%
%
%

\section{Concluding Remarks}

In this paper we gave improved bounds on the size of codes and families of permutations with a forbidden distance. These bounds demonstrate the power of  dependent random choice in forbidden distance problems and we expect that the method will have many more applications in extremal set theory.

It remains an intriguing open problem to obtain a better upper bound on the size of maximum $l$-avoiding families ${\mathcal A} \subset {\mathcal P}[n]$.  A natural construction is to take all sets that are `large' or `small', where `large' sets have size at least $(n+l)/2$ and `small' sets have size less than $l$.
(If $n+l$ is odd we can also add all sets of size $(n+l-1)/2$ containing $1$).
For fixed $l$ and large $n$, Frankl and F\"uredi \cite{FrFu} proved that this is the unique extremal family.

However, much less is known when $l$ is comparable with $n$. Under the stronger condition of  being $(l+1)$-intersecting, Katona \cite{ka} showed that the family of all large sets gives the optimal construction. Mubayi and R\"odl \cite{MR} conjectured that for the $l$-avoiding problem, with any $\epsilon n < l < (1/2-\epsilon)n$, the same family of all large sets and all small sets as before should be approximately optimal, say up to a multiplicative factor of $2^{o(n)}$.
They proved this when the $l$-avoiding condition is replaced with the stronger condition of having a small forbidden interval of intersections around $l$.

%
%
%
%
%
%
%
%

\end{document}